\documentclass{amsart}
\usepackage{amsfonts,tabularx,graphicx}
\usepackage{blkarray}
\usepackage{float}
\usepackage{tikz,url,hyperref}
\usepackage[a4paper, total={6.5in, 8.5in}]{geometry}
\newtheorem{theorem}{Theorem}[section]

\newtheorem{proposition}[theorem]{Proposition}

\theoremstyle{definition}
\newtheorem{definition}[theorem]{Definition}

\theoremstyle{remark}
\newtheorem{remark}[theorem]{Remark}
\numberwithin{equation}{section}
\newcommand{\F}{\mathbb{F}}

\begin{document}
\title[Generalized Neighbors]{New Extremal binary self-dual codes of length 68 from a novel approach to neighbors}
\author{
Joe Gildea, Abidin Kaya, Adrian Korban and Bahattin Yildiz}

\address{Department of Mathematical and Physical Sciences, Faculty of Science and Engineering, University of Chester, England, UK}
\email{j.gildea@chester.ac.uk}

\address{Department of Mathematics Education, Sampoerna University, 12780, Jakarta, Indonesia}
\email{abidin.kaya@sampoernauniversity.ac.id}

\address{Department of Mathematical and Physical Sciences, Faculty of Science and Engineering, University of Chester, England, UK}
\email{adrian3@windowslive.com}

\address{Department of Mathematics \& Statistics, Northern Arizona University, Flagstaff, AZ 86001, USA}
\email{bahattinyildiz@gmail.com}


\begin{abstract}
In this work, we introduce the concept of distance between self-dual
codes, which generalizes the concept of a neighbor for self-dual
codes. Using the $k$-neighbors, we are able to construct extremal
binary self-dual codes of length 68 with new weight enumerators. We
construct 143 extremal binary self-dual codes of length 68 with new
weight enumerators including 42 codes with $\gamma=8$ in their
$W_{68,2}$ and 40 with $\gamma=9$ in their $W_{68,2}$. These
examples are the first in the literature for these $\gamma$ values.
This completes the theoretical list of possible values for $\gamma$
in $W_{68,2}$.
\end{abstract}

\maketitle

\section{Introduction}
Self-dual codes are a special class of linear codes. Because of the
many interesting properties that they have and the many different
fields that they are connected with, they have attracted a
considerable interest in coding theory research community.

One of the most active research areas in the field of self-dual
codes is the construction and classification of extremal binary
self-dual codes. Type I extremal binary self-dual codes of lengths
such as 64, 66, 68, etc. have parameters in their weight
enumerators, which have not all been found to exist. Hence, the
rcent years have seen a surge of activity in finding extremal binary
self-dual codes of various lengths with new weight enumerators. Many
different techniques have been employed in constructing these
extremal binary self-dual codes such as constructions over certain
rings, constructions through automorphism groups, neighboring
constructions, shadows, extensions, etc. \cite{anev},
\cite{buyuklieva}, \cite{QR}, \cite{GKTY}, \cite{alteredFC},
\cite{kaya}, \cite{pasa} are just a sample of the works that contain
these ideas and their applications in finding new extremal binary
self-dual codes.

In this work, we introduce the concept of ``distance" between
self-dual codes. After proving some theoretical results about the
distance we observe that the neighbor can be defined in terms of the
distance and this leads to the concept of ``$k$-range neighbors" or
``$k$-neighbors", which generalize the concept of neighbors of
self-dual codes. We then use these $k$-neighbors to construct
extremal self-dual codes of length 68 from a given self-dual code.
In particular we construct 139 new extremal binary self-dual codes
of length 68 with new weight enumerators, including the first
examples with $\gamma=8, 9$ in $W_{68,2}$ in the literature. This
completes the list of possible $\gamma$ values that can be found in
$W_{68,2}$. 42 of the codes we have constructed have $\gamma=8$ in
their weight enumerator, while 40 of them have $\gamma=9$ in their
weight enumerators.

The rest of the work is organized as follows: In section 2, we give
the preliminaries about self-dual codes and the neighbor
construction.In section 3, we introduce the concept of distance and
define the related generalization of the neighbors. In section 4, we
apply the generalized neighbors to construct extremal binary
self-dual codes of length 68 with new weight enumerators. We finish
the work with concluding remarks and directions for possible future
research.
\section{Preliminaries}
\subsection{Self-dual codes}

For $\overline{x} = (x_1,x_2, \dots, x_n)$ and $\overline{y} =
(y_1,y_2, \dots, y_n) \in \F_2^n$, we define
$$
\langle\overline{x}, \overline{y}\rangle = x_1y_1+x_2y_2+ \dots +
x_ny_n.
$$
This inner product leads to

\begin{definition}
Let $C$ be a binary linear code over  of length $n$, then we define
the {dual} of $C$ as
$$C^{\perp}:= \large\{ \overline{y} \in \F_2^n \large | \langle \overline{y},\overline{x}\rangle = 0, \:\:\:\: \forall
\overline{x} \in C \large \}.$$
\end{definition}

 Note that, if $C$ is a linear $[n,k]$-code, then $C^{\perp}$ is a linear $[n,n-k]$-code.

\begin{definition}
If $C \subseteq C^{\perp}$, then $C$ is called {\it self-orthogonal}
and it is called {\it self-dual} if $C = C^{\perp}.$
\end{definition}

\begin{definition}
Let $C$ be a self-dual binary code. If the Hamming weights of all
the codewords in $C$ are divisible by $4$, $C$ is called {\bf Type
II} (or doubly-even), otherwise it is called {\bf Type I} (or singly
even).
\end{definition}

The following theorem gives an upper bound for minimum distance of
self-dual codes:

\begin{theorem}$($\cite{Rains}, \cite{conway}$)$ Let $d_I(n)$ and $d_{II}(n)$ be the minimum distance of a Type I and Type II binary code of length $n$. then
$$d_{II}(n) \leq 4 \lfloor\frac{n}{24}\rfloor+4$$ and
$$d_{I}(n) \leq \left \{
\begin{array}{ll}
 4 \lfloor\frac{n}{24}\rfloor+4 & \textrm{if $n \not \equiv 22 \pmod{24}$}
\\
4 \lfloor\frac{n}{24}\rfloor+6 & \textrm {if $n \equiv 22
\pmod{24}$.}
\end{array}\right.$$
\end{theorem}

Self-dual codes that attain the bounds given in the previous theorem
are called {\it extremal}.

\subsection{The neighbor construction}
Two self-dual codes of length $n$ are called {\it neighbors} if
their intersection is a code of dimension $\frac{n}{2}-1$. This idea
has been used extensively in the literature  to construct new
self-dual codes from an existing one. For some of the works that
have used this idea, we can refer to \cite{brualdi}, \cite{QR},
\cite{alteredFC} and references therein.

Given a self-dual code $C$, a vector $x\in {\mathbb{F}}_{2}^{n}-C$ is picked and then $D$ is formed by letting $%
D=\left\langle \left\langle x\right\rangle ^{\bot }\cap
C,x\right\rangle $. The search for $D$ can be made efficient by
using the standard form of the generator matrix of $C$, which lets
one to fix the first $n/2$ entries of $x$ without loss of
generality.  Usually in practical applications the first $n/2$
entries of $x$ are set to be $0$.

\section{Distance between self-dual codes and generalized neighbors}
To generalize the notion of a neighbor, we first begin with the
following definition of a distance between two self-dual codes:
\begin{definition}
Let $C_1$ and $C_2$ be two binary self-dual codes of length $n$. The
{\it neighbor-distance} between $C_1$ and $C_2$ is defined as
$$d_N(C_1, C_2) = \frac{n}{2}-dim(C_1\cap C_2).$$
\end{definition}

\begin{proposition}
$d_N$ is a metric on the set of all binary self-dual codes of length
$n$.
\end{proposition}

\begin{proof}
Since $dim(C_1\cap C_2) \leq dim(C_1)=dim(C_2) =\frac{n}{2}$, we
have $d_N(C_1, C_2) \geq 0$ for all self-dual codes $C_1, C_2$.

Next, observe that if $d_N(C_1, C_2) = 0$, this means
$$dim(C_1\cap C_2) = \frac{n}{2} = dim(C_1)=dim(C_2),$$
which implies $C_1=C_2=C_1\cap C_2$. Conversely, if $C_1=C_2$, then
$d_N(C_1,C_2) = \frac{n}{2}-\frac{n}{2}=0.$

By the definition, it is clear that $d_N(C_1,C_2) = d_N(C_2, C_1)$.

For the triangle inequality, assume that $C_1, C_2, C_3$ are
self-dual codes. Observe that
$$(C_1\cap C_2) \cup (C_2\cap C_3) = C_2\cap(C_1\cup C_3) \subseteq C_2.$$
which implies
$$dim(C_1\cap C_2)+dim(C_2\cap C_3) \leq dim(C_2) = \frac{n}{2}.$$
Thus we have
$$dim(C_1\cap C_2)+dim(C_2\cap C_3) -dim(C_1\cap C_3) \leq \frac{n}{2}.$$
Adding $\frac{n}{2}$ to both sides and sending $dim(C_1\cap
C_2)+dim(C_2\cap C_3)$ over to the right side of the equation, we
get
\begin{align*}
    d_N(C_1, C_3) & =  \frac{n}{2}-dim(C_1\cap C_3) \\
    & \leq \frac{n}{2}-dim(C_1\cap C_2)+\frac{n}{2}-dim(C_2\cap C_3)\\
    & \leq d_N(C_1, C_2)+d_N(C_2, C_3).
\end{align*}
\end{proof}

The next proposition shows that self-dual codes cannot have the
maximum distance to each other:
\begin{proposition}
Let $C_1$ and $C_2$ be two binary self-dual codes of length $n$.
Then $d_N(C_1,C_2) < \frac{n}{2}$.
\end{proposition}

\begin{proof}
It is well known that if $C$ is any binary self-dual code, then
$(1,1, \dots, 1) \in C$. thus $\overline{1} \in C_1\cap C_2$, which
implies that $dim(C_1\cap C_2) \geq 1$. But then
$$d_N(C_1, C_2) = \frac{n}{2}-dim(C_1\cap C_2) \leq \frac{n}{2}-1 < \frac{n}{2}.$$
\end{proof}

\bigskip
\noindent {\bf Question:} Is there an upper bound on the distance
between two self-dual codes? The proposition shows that the distance
cannot be larger than $\frac{n}{2}-1$. It is an open question
whether this upper bound can be reduced further.

\bigskip

We now define {\it $k$-range neighbor} and {\it $k$-neighbor} of a
code:
\begin{definition}
Let $C_1$ and $C_2$ be two self-dual codes. $C_1$ and $C_2$ are said
to be {\it $k$-range neighbors} if $d_N(C_1, C_2) \leq k$ and they
are called {\it $k$-neighbors} if $d_N(C_1,C_2)=k$.
\end{definition}

\begin{remark}
The {\it neighbor} of a self-dual code is well known in the
literature and it corresponds to a $1$-neighbor in our context.
\end{remark}

\begin{remark}
The concept of a $k$-range neighbor code can be more useful than the
strict $k$-neighbor codes, because of the following observation:

Suppose $C_1$ and $C_2$ are self-dual binary codes with generator
matrices $[I_{n/2}|M_1]$ and $[I_{n/2}|M_2]$, respectively, where
$$M_1 = \left[
\begin{array}{c}
\overline{r}_1 \\
\hline
\overline{r}_2 \\
\hline
\overline{r}_3 \\
\hline
\vdots \\
\hline \overline{r}_{n/2}
\end{array}
\right], \:\:\:\:\:\:\:\:\:\:\:\:\:\: M_2 = \left[
\begin{array}{c}
\overline{s}_1 \\
\hline
\overline{s}_2 \\
\hline
\overline{s}_3 \\
\hline
\vdots \\
\hline \overline{s}_{n/2}
\end{array}
\right].$$ Here $\overline{r}_i$ and $\overline{s}_j$ are the rows
of $M_1$ and $M_2$ respectively. If $\overline{r}_{i} =
\overline{s}_i$ for $i=k+1, k+2, \dots, n/2$, then $C_1$ and $C_2$
are $k$-range neighbors.
\end{remark}

\begin{remark}\label{rem}
As we observed above, the ordinary neighbor of a code $C$ is a
$1$-neighbor. We can also observe that, the neighbor a $1$-neighbor
of $C$ is a $2$-range neighbor of $C$. This can be generalized into
considering the neighbor of a neighbor of a neighbor etc. of a code
as a $k$-range neighbor of the original code.
\end{remark}

\section{Applications of $k$-range neighbor codes to extremal self-dual codes}
In this section we will give an equivalent description for the
$k$-range neighbors and use them to construct new extremal binary
self-dual codes. Let $\mathcal{N}_{(0)}$ be a binary self-dual code
of length $2n$. Let $x_0 \in \F_{2}^{2n}-\mathcal{N}_{(0)}$, define

\[
\mathcal{N}_{(i+1)}=\left\langle \left\langle x_i \right\rangle
^{\bot }\cap \mathcal{N}_{(i)},x_i\right\rangle
\]

\noindent where $\mathcal{N}_{(i+1)}$ is the  neighbour of
$\mathcal{N}_{(i)}$ and $x_i \in \F_{2}^{2n}-\mathcal{N}_{(i)}$.

It is not hard to see that $\mathcal{N}_{(i)}$ defined in this way
is an $i$-range neighbor of $\mathcal{N}_{(0)}$ as was observed
above in Remark \ref{rem}. In what follows, we will apply this idea
to search for extremal binary self-dual codes from $k$-range
neighbors of a known code. We use Magma Algebra System
(\cite{magma}) for our searches.

\subsection{Numerical results from $i$-range neighbours}

The possible weight enumerator of an extremal binary self-dual code
of length 68 (of parameters $\left[ 68,34,12\right]$) is in one of
the following forms by \cite{buyuklieva,harada, dougherty1}:
\begin{eqnarray*}
W_{68,1} &=&1+\left( 442+4\beta \right) y^{12}+\left( 10864-8\beta
\right)
y^{14}+\cdots ,104\leq \beta \leq 1358, \\
W_{68,2} &=&1+\left( 442+4\beta \right) y^{12}+\left( 14960-8\beta
-256\gamma \right) y^{14}+\cdots
\end{eqnarray*}%
where $0\leq \gamma \leq 9$. Recently, Yankov et al. constructed the
first examples of codes with a weight enumerator for $\gamma =7$ in
$W_{68,2}$ in \cite{anev}. Together with these, the existence of
codes in $W_{68,2}$ is known for many values. In order to save space
we only give the lists for $\gamma =5$, $\gamma =6$ and  $\gamma
=7$, which are updated in this work;
\[
\begin{split}
\gamma &=5\text{ with }\beta \in \left\{\text{101,105,109,111,$...$,182,187,189,191,192,193,201,202,213}\right\}\\
\gamma &=6\text{ with }\beta \in \left\{ 133,137,139,\ldots,174,176,177,184,192,210\right\}\\
\gamma  &=7\text{ with }\beta \in \left\{
7m|m=14,\ldots,39,42\right\}
\end{split}
\]

\noindent Let $\mathcal{N}_{(0)}$ be the extremal binary self-dual
code of length $68$ ($W_{68,2}$) with the parameters $\gamma=5$ and
$\beta=213$ which was recently constructed in
\cite{JoeAbidinAdrian}. Its binary generating matrix is given by
$(I_{34}|A)$ where

 $$A= \left( \begin{smallmatrix}
0&0&0&1&0&0&0&1&1&1&1&0&1&1&1&1&1&1&0&0&1&1&1&1&0&1&0&1&0&0&0&1&1&0\\
0&1&1&0&0&0&1&0&0&1&1&1&1&0&1&0&1&0&0&0&0&1&1&0&1&0&0&0&1&1&0&0&1&0\\
0&1&0&0&1&0&1&0&1&1&0&1&0&0&0&0&1&0&1&1&0&1&1&1&1&0&1&0&0&0&0&0&0&1\\
1&0&0&0&0&1&0&1&0&0&0&1&1&1&1&1&1&0&1&1&0&1&0&0&0&1&1&0&0&0&1&1&1&0\\
0&0&1&0&1&1&1&0&0&1&1&1&1&0&0&0&1&0&1&1&1&0&1&0&1&0&0&0&0&0&0&1&1&1\\
0&0&0&1&1&1&0&1&1&0&1&1&0&1&0&0&0&1&1&1&0&1&0&1&0&1&0&0&0&0&1&0&1&1\\
0&1&1&0&1&0&0&0&0&0&1&1&0&1&1&1&1&0&1&0&0&0&1&1&0&0&1&0&1&0&0&1&0&0\\
1&0&0&1&0&1&0&0&1&1&0&0&0&1&0&0&1&0&0&1&1&1&0&0&0&0&1&0&0&1&0&1&0&0\\
1&0&1&0&0&1&0&1&0&1&1&0&0&0&0&0&0&1&0&1&0&0&1&0&0&0&0&0&1&0&1&0&1&1\\
1&1&1&0&1&0&0&1&1&1&1&1&1&0&1&0&0&0&1&0&0&1&1&1&1&1&1&0&1&0&1&1&1&1\\
1&1&0&0&1&1&1&0&1&1&1&0&0&0&1&1&0&1&1&1&0&0&1&1&0&1&1&1&0&1&0&1&1&0\\
0&1&1&1&1&1&1&0&0&1&0&0&0&1&1&0&1&1&1&1&1&0&1&0&0&1&1&0&0&1&1&1&0&1\\
1&1&0&1&1&0&0&0&0&1&0&0&1&0&1&0&1&1&1&1&0&0&0&1&0&1&0&1&1&0&0&1&0&1\\
1&0&0&1&0&1&1&1&0&0&0&1&0&0&0&0&1&0&1&1&1&0&1&1&0&1&1&1&1&0&1&1&1&0\\
1&1&0&1&0&0&1&0&0&1&1&1&1&0&1&0&0&0&0&1&0&1&0&0&1&1&0&1&0&1&1&0&0&1\\
1&0&0&1&0&0&1&0&0&0&1&0&0&0&0&0&0&1&1&0&0&0&0&1&0&0&1&1&0&1&0&0&1&0\\
1&1&1&1&1&0&1&1&0&0&0&1&1&1&0&0&1&0&0&1&1&1&0&0&1&1&1&0&0&1&0&1&1&1\\
1&0&0&0&0&1&0&0&1&0&1&1&1&0&0&1&0&0&1&0&0&1&0&1&0&0&0&0&0&1&1&1&1&1\\
0&0&1&1&1&1&1&0&0&1&1&1&1&1&0&1&0&1&0&0&0&1&1&1&0&0&1&1&1&0&0&0&1&0\\
0&0&1&1&1&1&0&1&1&0&1&1&1&1&1&0&1&0&0&0&1&0&1&1&0&0&1&1&0&1&0&0&0&1\\
1&0&0&0&1&0&0&1&0&0&0&1&0&1&0&0&1&0&0&1&1&1&0&1&0&0&1&0&0&1&0&0&0&1\\
1&1&1&1&0&1&0&1&0&1&0&0&0&0&1&0&1&1&1&0&1&0&0&0&1&1&1&1&0&1&1&0&1&0\\
1&1&1&0&1&0&1&0&1&1&1&1&0&1&0&0&0&0&1&1&0&0&0&1&0&1&0&1&0&1&1&0&1&1\\
1&0&1&0&0&1&1&0&0&1&1&0&1&1&0&1&0&1&1&1&1&0&1&1&0&1&1&1&0&1&0&0&1&1\\
0&1&1&0&1&0&0&0&0&1&0&0&0&0&1&0&1&0&0&0&0&1&0&0&1&1&1&1&1&1&1&1&1&1\\
1&0&0&1&0&1&0&0&0&1&1&1&1&1&1&0&1&0&0&0&0&1&1&1&1&1&0&0&1&1&0&0&1&1\\
0&0&1&1&0&0&1&1&0&1&1&1&0&1&0&1&1&0&1&1&1&1&0&1&1&0&0&1&0&1&1&0&1&1\\
1&0&0&0&0&0&0&0&0&0&1&0&1&1&1&1&0&0&1&1&0&1&1&1&1&0&1&1&0&1&0&0&1&1\\
0&1&0&0&0&0&1&0&1&1&0&0&1&1&0&0&0&0&1&0&0&0&0&0&1&1&0&0&1&1&1&0&1&0\\
0&1&0&0&0&0&0&1&0&0&1&1&0&0&1&1&1&1&0&1&1&1&1&1&1&1&1&1&1&1&1&0&0&1\\
0&0&0&1&0&1&0&0&1&1&0&1&0&1&1&0&0&1&0&0&0&1&1&0&1&0&1&0&1&1&0&0&0&1\\
1&0&0&1&1&0&1&1&0&1&1&1&1&1&0&0&1&1&0&0&0&0&0&0&1&0&0&0&0&0&0&1&1&0\\
1&1&0&1&1&1&0&0&1&1&1&0&0&1&0&1&1&1&1&0&0&1&1&1&1&1&1&1&1&0&0&1&0&1\\
0&0&1&0&1&1&0&0&1&1&0&1&1&0&1&0&1&1&0&1&1&0&1&1&1&1&1&1&0&1&1&0&1&0
\end{smallmatrix}\right).$$

\bigskip

Implementing the formula described above to this code
$\mathcal{N}_{(0)}$, we obtain:

\begin{table}[H]\caption{$i$-range neighbour of  $\mathcal{N}_{(0)}$}\label{neighbors1}
\begin{center}\scalebox{0.9}{
\begin{tabular}{ccccccc}
\hline $i$ & $\mathcal{N}_{(i+1)}$   & $x_i$ &
$|Aut(\mathcal{N}_{(i+1)})   |$ & $\gamma$ & $\beta$  \\ \hline
\hline $0$ & $\mathcal{N}_{(1)}$ &
$(1100000101101111011001110100000100)$ & $1$ & $6$ & $210$   \\
\hline $1$ & $\mathcal{N}_{(2)}$ &
$(0111111111110110010100110111001100)$ & $1$ & $\textbf{7}$ &
$\textbf{212}$   \\ \hline $2$ & $\mathcal{N}_{(3)}$ &
$(0111010010010101001000101110011001)$ & $1$ & $\textbf{8}$ &
$\textbf{221}$   \\ \hline $3$ & $\mathcal{N}_{(4)}$ &
$(1000000111110011101001110001110000)$ & $1$ & $\textbf{9}$ &
$\textbf{221}$   \\ \hline
\end{tabular}}
\end{center}
\end{table}

\subsection{Neighbours of Neighbours}

\noindent In this section, we separately consider neighbours of
$\mathcal{N}_{(0)}$, $\mathcal{N}_{(1)}$, $\mathcal{N}_{(2)}$,
$\mathcal{N}_{(3)}$ and $\mathcal{N}_{(4)}$.

\begin{table}[H]\caption{Neighbours of  $\mathcal{N}_{(0)}$}\label{neighbors1}
\begin{center}\scalebox{0.9}{
\begin{tabular}{cccccc}
\hline $\mathcal{C}_{i}$   & $(x_{35},x_{36},...,x_{68})$ &
$|Aut(\mathcal{N}_{68,i})   |$ & $\gamma$ & $\beta$  \\ \hline
\hline $\mathcal{C}_{1}$ &  $(1001100000010100010100001111100011)$ &
$2$ & $\textbf{5}$ & $\textbf{195}$  \\ \hline $\mathcal{C}_{2}$ &
$(1000010001011010000011010000011010)$ & $1$ & $\textbf{5}$ &
$\textbf{198}$  \\ \hline $\mathcal{C}_{3}$ &
$(0111101000110110001011101100010000)$ & $1$ & $\textbf{5}$ &
$\textbf{200}$  \\ \hline $\mathcal{C}_{4}$ &
$(0111001101010010011001000101101010)$ & $1$ & $\textbf{5}$ &
$\textbf{202}$  \\ \hline $\mathcal{C}_{5}$ &
$(0100101101000111111110110101110111)$ & $2$ & $\textbf{5}$ &
$\textbf{211}$  \\ \hline $\mathcal{C}_{6}$ &
$(0011011100110001100010000000100100)$ & $1$ & $\textbf{6}$ &
$\textbf{198}$  \\ \hline $\mathcal{C}_{7}$ &
$(0111011111101001111101101111001000)$ & $1$ & $\textbf{6}$ &
$\textbf{204}$  \\ \hline
\end{tabular}}
\end{center}
\end{table}

\begin{table}[H]\caption{Neighbours of  $\mathcal{N}_{(1)}$}\label{neighbors1}
\begin{center}\scalebox{0.9}{
\begin{tabular}{cccc|cccccc}
\hline $\mathcal{C}_{i}$   & $(x_{35},x_{36},...,x_{68})$  &
$\gamma$ & $\beta$ & $\mathcal{C}_{i}$   &
$(x_{35},x_{36},...,x_{68})$  & $\gamma$ & $\beta$ \\ \hline  \hline
$\mathcal{C}_{8}$ &  $(1001010111010111110011100111000011)$  &
$\textbf{6}$ & $\textbf{175}$   & $\mathcal{C}_{9}$ &
$(0001011110111110011101001111111100)$  & $\textbf{6}$ &
$\textbf{177}$   \\ \hline $\mathcal{C}_{10}$ &
$(1011110110111010111010010111101111)$  & $\textbf{6}$ &
$\textbf{179}$   & $\mathcal{C}_{11}$ &
$(1011011001101100010101001010001111)$  & $\textbf{6}$ &
$\textbf{181}$   \\ \hline $\mathcal{C}_{12}$ &
$(0111001000010101110001001100111100)$  & $\textbf{6}$ &
$\textbf{182}$   & $\mathcal{C}_{13}$ &
$(0111111111011111101100100100001110)$  & $\textbf{6}$ &
$\textbf{183}$   \\ \hline $\mathcal{C}_{14}$ &
$(1011111001001110011110000010100011)$  & $\textbf{6}$ &
$\textbf{185}$   & $\mathcal{C}_{15}$ &
$(1010100011111100010011111101001101)$  & $\textbf{6}$ &
$\textbf{186}$   \\ \hline $\mathcal{C}_{16}$ &
$(1011010111110011001011000100111011)$  & $\textbf{6}$ &
$\textbf{187}$   & $\mathcal{C}_{17}$ &
$(0011110011110111111101101100110100)$  & $\textbf{6}$ &
$\textbf{188}$   \\ \hline $\mathcal{C}_{18}$ &
$(0000000010010101010011001010001011)$  & $\textbf{6}$ &
$\textbf{189}$   & $\mathcal{C}_{19}$ &
$(1011101110011101111110100101011100)$  & $\textbf{6}$ &
$\textbf{190}$   \\ \hline $\mathcal{C}_{20}$ &
$(1111000000010110001111001111010101)$  & $\textbf{6}$ &
$\textbf{191}$   & $\mathcal{C}_{21}$ &
$(1010111010101011100011100011001111)$  & $\textbf{6}$ &
$\textbf{193}$   \\ \hline $\mathcal{C}_{22}$ &
$(1000001101010100110101000011000101)$  & $\textbf{6}$ &
$\textbf{194}$   & $\mathcal{C}_{23}$ &
$(0101011110100101000000011010001111)$  & $\textbf{6}$ &
$\textbf{195}$   \\ \hline $\mathcal{C}_{24}$ &
$(1111011101111101110110100010000111)$  & $\textbf{6}$ &
$\textbf{196}$   & $\mathcal{C}_{25}$ &
$(1111011010010110011101100001000110)$  & $\textbf{6}$ &
$\textbf{197}$   \\ \hline $\mathcal{C}_{26}$ &
$(1011110001001000100001000110100000)$  & $\textbf{6}$ &
$\textbf{199}$   & $\mathcal{C}_{27}$ &
$(1001010010110100100000001100000101)$  & $\textbf{6}$ &
$\textbf{200}$   \\ \hline $\mathcal{C}_{28}$ &
$(1010010010100011111100011100111010)$  & $\textbf{6}$ &
$\textbf{201}$   & $\mathcal{C}_{29}$ &
$(0100010011101100010110001010110000)$  & $\textbf{6}$ &
$\textbf{202}$   \\ \hline $\mathcal{C}_{30}$ &
$(0100100000011110000010011000010110)$  & $\textbf{6}$ &
$\textbf{206}$   & $\mathcal{C}_{31}$ &
$(0001110111010001111011010001011111)$  & $\textbf{6}$ &
$\textbf{207}$   \\ \hline $\mathcal{C}_{32}$ &
$(0010110010011001111110101000011110)$  & $\textbf{7}$ &
$\textbf{184}$   & $\mathcal{C}_{33}$ &
$(0011001100100010111110000000011001)$  & $\textbf{7}$ &
$\textbf{185}$   \\ \hline
\end{tabular}}
\end{center}
\end{table}

\begin{table}[H]\caption{Neighbours of  $\mathcal{N}_{(2)}$}\label{neighbors2}
\begin{center}\scalebox{0.9}{
\begin{tabular}{cccc|cccccc}
\hline $\mathcal{C}_{i}$   & $(x_{35},x_{36},...,x_{68})$  &
$\gamma$ & $\beta$ & $\mathcal{C}_{i}$   &
$(x_{35},x_{36},...,x_{68})$  & $\gamma$ & $\beta$  \\ \hline
\hline $\mathcal{C}_{34}$ &  $(1101111111101111011110011101101110)$
& $\textbf{7}$ & $\textbf{174}$   & $\mathcal{C}_{35}$ &
$(1001101001000111010110110111111111)$  & $\textbf{7}$ &
$\textbf{177}$   \\ \hline $\mathcal{C}_{36}$ &
$(1011010010001010111100010010000100)$  & $\textbf{7}$ &
$\textbf{178}$   & $\mathcal{C}_{37}$ &
$(1101111001010110110111010110001111)$  & $\textbf{7}$ &
$\textbf{179}$   \\ \hline $\mathcal{C}_{38}$ &
$(1011101100100101100110111101101111)$  & $\textbf{7}$ &
$\textbf{181}$   & $\mathcal{C}_{39}$ &
$(1001111111011111110110010001001110)$  & $\textbf{7}$ &
$\textbf{183}$   \\ \hline $\mathcal{C}_{40}$ &
$(1101010001010110001001111001100010)$  & $\textbf{7}$ &
$\textbf{186}$   & $\mathcal{C}_{41}$ &
$(0110111111000000011011000001110001)$  & $\textbf{7}$ &
$\textbf{187}$   \\ \hline $\mathcal{C}_{42}$ &
$(0010011000010000000011111111111010)$  & $\textbf{7}$ &
$\textbf{188}$   & $\mathcal{C}_{42}$ &
$(1001101101101011111111010100101101)$  & $\textbf{7}$ &
$\textbf{190}$   \\ \hline $\mathcal{C}_{43}$ &
$(0110000001011010011001111110100010)$  & $\textbf{7}$ &
$\textbf{191}$   & $\mathcal{C}_{44}$ &
$(1011000011111010100011111011100011)$  & $\textbf{7}$ &
$\textbf{192}$   \\ \hline $\mathcal{C}_{45}$ &
$(0100001110010111010110101010011110)$  & $\textbf{7}$ &
$\textbf{193}$   & $\mathcal{C}_{46}$ &
$(0000001101111110100001100001000000)$  & $\textbf{7}$ &
$\textbf{194}$   \\ \hline $\mathcal{C}_{47}$ &
$(0100110110001011101001011000110001)$  & $\textbf{7}$ &
$\textbf{195}$   & $\mathcal{C}_{48}$ &
$(1111010100000010111100100101110101)$  & $\textbf{7}$ &
$\textbf{197}$   \\ \hline $\mathcal{C}_{49}$ &
$(1100100000001001100110010111111111)$  & $\textbf{7}$ &
$\textbf{198}$   & $\mathcal{C}_{50}$ &
$(0000110011100001111010110100110001)$  & $\textbf{7}$ &
$\textbf{199}$   \\ \hline
\end{tabular}}
\end{center}
\end{table}

\begin{table}[H]\caption{Neighbours of  $\mathcal{N}_{(3)}$}\label{neighbors3}
\begin{center}\scalebox{0.8}{
\begin{tabular}{cccc|cccccc}
\hline $\mathcal{C}_{i}$   & $(x_{35},x_{36},...,x_{68})$ &
$\gamma$ & $\beta$ &  $\mathcal{C}_{i}$   &
$(x_{35},x_{36},...,x_{68})$ &  $\gamma$ & $\beta$ \\ \hline  \hline
$\mathcal{C}_{51}$ &  $(0101000010100010000111100101011100)$  &
$\textbf{7}$ & $\textbf{171}$   & $\mathcal{C}_{52}$ &
$(0000011101001011001010111100001110)$  & $\textbf{7}$ &
$\textbf{173}$   \\ \hline $\mathcal{C}_{53}$ &
$(0000011101101010010000110000101000)$  & $\textbf{7}$ &
$\textbf{176}$   & $\mathcal{C}_{54}$ &
$(1011010010100001010000111011100110)$  & $\textbf{7}$ &
$\textbf{180}$   \\ \hline $\mathcal{C}_{55}$ &
$(1011110011111111111100111110100111)$  & $\textbf{8}$ &
$\textbf{181}$   & $\mathcal{C}_{56}$ &
$(1101111111010111111111100010110111)$  & $\textbf{8}$ &
$\textbf{186}$   \\ \hline $\mathcal{C}_{57}$ &
$(0000110100100011011001001101111010)$  & $\textbf{8}$ &
$\textbf{187}$   & $\mathcal{C}_{58}$ &
$(1101000111000011001010010001000000)$  & $\textbf{8}$ &
$\textbf{189}$   \\ \hline $\mathcal{C}_{59}$ &
$(0000001011100100110100101111000100)$  & $\textbf{8}$ &
$\textbf{190}$   & $\mathcal{C}_{60}$ &
$(0111011011011011010010101101100011)$  & $\textbf{8}$ &
$\textbf{192}$   \\ \hline $\mathcal{C}_{61}$ &
$(0011011011110001000111111111011110)$  & $\textbf{8}$ &
$\textbf{191}$   & $\mathcal{C}_{62}$ &
$(1010111011000111100110111001110111)$  & $\textbf{8}$ &
$\textbf{193}$   \\ \hline $\mathcal{C}_{63}$ &
$(0010010011101001110101011010111100)$  & $\textbf{8}$ &
$\textbf{194}$   & $\mathcal{C}_{64}$ &
$(0001001001010000111101111001110111)$  & $\textbf{8}$ &
$\textbf{195}$   \\ \hline $\mathcal{C}_{65}$ &
$(0001101011101101100010111110110011)$  & $\textbf{8}$ &
$\textbf{196}$   & $\mathcal{C}_{66}$ &
$(0100100001010000001010001111100010)$  & $\textbf{8}$ &
$\textbf{197}$   \\ \hline $\mathcal{C}_{67}$ &
$(1011100101000101100011110111101011)$  & $\textbf{8}$ &
$\textbf{198}$   & $\mathcal{C}_{68}$ &
$(0011011111101011011011111011110011)$  & $\textbf{8}$ &
$\textbf{199}$   \\ \hline $\mathcal{C}_{69}$ &
$(1111000101100111100010101010000001)$  & $\textbf{8}$ &
$\textbf{200}$   & $\mathcal{C}_{70}$ &
$(0011000110010100110010000110000001)$  & $\textbf{8}$ &
$\textbf{201}$   \\ \hline $\mathcal{C}_{71}$ &
$(1110101001000010010100101000011100)$  & $\textbf{8}$ &
$\textbf{202}$   & $\mathcal{C}_{72}$ &
$(0110111010011110110001011011101001)$  & $\textbf{8}$ &
$\textbf{203}$   \\ \hline $\mathcal{C}_{73}$ &
$(1001100101111110111101011001101110)$  & $\textbf{8}$ &
$\textbf{204}$   & $\mathcal{C}_{74}$ &
$(0000100111111101000010110011001001)$  & $\textbf{8}$ &
$\textbf{205}$   \\ \hline $\mathcal{C}_{75}$ &
$(1011001000010010011100101011000100)$  & $\textbf{8}$ &
$\textbf{206}$   & $\mathcal{C}_{76}$ &
$(0101111110001111110000111111111011)$  & $\textbf{8}$ &
$\textbf{207}$   \\ \hline $\mathcal{C}_{77}$ &
$(0110101010100001110101011010110110)$  & $\textbf{8}$ &
$\textbf{208}$   & $\mathcal{C}_{78}$ &
$(0001111000110101011111001111101111)$  & $\textbf{8}$ &
$\textbf{209}$   \\ \hline $\mathcal{C}_{79}$ &
$(1111011101111110000011100111111011)$  & $\textbf{8}$ &
$\textbf{210}$   & $\mathcal{C}_{80}$ &
$(1101010100000100000001110100010001)$  & $\textbf{8}$ &
$\textbf{211}$   \\ \hline $\mathcal{C}_{81}$ &
$(0011110101110001000001111001110000)$  & $\textbf{8}$ &
$\textbf{212}$  & $\mathcal{C}_{82}$ &
$(1100011110110111110101000101011111)$  & $\textbf{8}$ &
$\textbf{213}$   \\ \hline $\mathcal{C}_{83}$ &
$(0101011001011011111001010100001000)$  & $\textbf{8}$ &
$\textbf{214}$   & $\mathcal{C}_{84}$ &
$(0000011110101100110001010101100011)$  & $\textbf{8}$ &
$\textbf{215}$   \\ \hline $\mathcal{C}_{85}$ &
$(1110100101011111001101011011011110)$  & $\textbf{8}$ &
$\textbf{216}$   & $\mathcal{C}_{86}$ &
$(0001111000001111100010100011011010)$  & $\textbf{8}$ &
$\textbf{217}$   \\ \hline $\mathcal{C}_{87}$ &
$(1110011000000010100101000101010110)$  & $\textbf{8}$ &
$\textbf{218}$   & $\mathcal{C}_{88}$ &
$(0100011111001011000000000010000011)$  & $\textbf{8}$ &
$\textbf{220}$   \\ \hline
\end{tabular}}
\end{center}
\end{table}

\begin{table}[H]\caption{Neighbours of  $\mathcal{N}_{(4)}$}\label{neighbors4}
\begin{center}\scalebox{0.8}{
\begin{tabular}{cccccccccc}
\hline $\mathcal{C}_{i}$   & $(x_{35},x_{36},...,x_{68})$ &
$\gamma$ & $\beta$  &  $\mathcal{C}_{i}$   &
$(x_{35},x_{36},...,x_{68})$ &  $\gamma$ & $\beta$ \\ \hline  \hline
$\mathcal{C}_{89}$ &  $(0110011110100110101110111111001110)$  &
$\textbf{7}$ & $\textbf{163}$   & $\mathcal{C}_{90}$ &
$(1010011101011110101011111111011110)$  & $\textbf{7}$ &
$\textbf{166}$   \\ \hline $\mathcal{C}_{91}$ &
$(1001110111100010010000100001111010)$  & $\textbf{7}$ &
$\textbf{169}$   & $\mathcal{C}_{92}$ &
$(0011101011111100101001110010011011)$  & $\textbf{7}$ &
$\textbf{170}$   \\ \hline $\mathcal{C}_{93}$ &
$(1101101101010111100010000101001101)$  & $\textbf{7}$ &
$\textbf{172}$   & $\mathcal{C}_{94}$ &
$(1001001100010110100011110011101101)$  & $\textbf{8}$ &
$\textbf{180}$   \\ \hline $\mathcal{C}_{95}$ &
$(0110000111110011101010000111110111)$  & $\textbf{8}$ &
$\textbf{182}$   & $\mathcal{C}_{96}$ &
$(1110000100101011001100000100001101)$  & $\textbf{8}$ &
$\textbf{183}$   \\ \hline $\mathcal{C}_{97}$ &
$(1110111101111011001110111111010111)$  & $\textbf{8}$ &
$\textbf{184}$   & $\mathcal{C}_{98}$ &
$(0111001000001101101001110011010010)$  & $\textbf{8}$ &
$\textbf{185}$   \\ \hline $\mathcal{C}_{99}$ &
$(0010010100101110111101101011111111)$  & $\textbf{8}$ &
$\textbf{219}$   & $\mathcal{C}_{100}$ &
$(1101000000110100011000000110101000)$  & $\textbf{8}$ &
$\textbf{188}$   \\ \hline $\mathcal{C}_{101}$ &
$(0111101011001101101011010011001011)$  & $\textbf{9}$ &
$\textbf{186}$   & $\mathcal{C}_{102}$ &
$(0001001011000000110111110000010110)$  & $\textbf{9}$ &
$\textbf{187}$   \\ \hline $\mathcal{C}_{103}$ &
$(1011011001100001001110011100101101)$  & $\textbf{9}$ &
$\textbf{188}$   & $\mathcal{C}_{104}$ &
$(1000000111000110111000100010000001)$  & $\textbf{9}$ &
$\textbf{189}$   \\ \hline $\mathcal{C}_{105}$ &
$(1111010100111110101110110000011111)$  & $\textbf{9}$ &
$\textbf{190}$   & $\mathcal{C}_{106}$ &
$(1000100000000111110001110000010010)$  & $\textbf{9}$ &
$\textbf{192}$   \\ \hline $\mathcal{C}_{107}$ &
$(0111011010010110011110110001000110)$  & $\textbf{9}$ &
$\textbf{193}$   & $\mathcal{C}_{108}$ &
$(0010011100001000000010001111011000)$  & $\textbf{9}$ &
$\textbf{194}$   \\ \hline $\mathcal{C}_{109}$ &
$(0001101100111010100010011110101000)$  & $\textbf{9}$ &
$\textbf{195}$   & $\mathcal{C}_{110}$ &
$(1000011011111111111010001110010001)$  & $\textbf{9}$ &
$\textbf{196}$   \\ \hline $\mathcal{C}_{111}$ &
$(0111010111111001111101011000101110)$  & $\textbf{9}$ &
$\textbf{198}$   & $\mathcal{C}_{112}$ &
$(0101101101001100001001110011010010)$  & $\textbf{9}$ &
$\textbf{199}$   \\ \hline $\mathcal{C}_{113}$ &
$(1111011011111111111010100100111001)$  & $\textbf{9}$ &
$\textbf{200}$   & $\mathcal{C}_{114}$ &
$(0011111100000101110110110111011111)$  & $\textbf{9}$ &
$\textbf{201}$   \\ \hline $\mathcal{C}_{115}$ &
$(0100111111101001101001110001101011)$  & $\textbf{9}$ &
$\textbf{202}$   & $\mathcal{C}_{116}$ &
$(1111101111000110001100111111101100)$  & $\textbf{9}$ &
$\textbf{203}$   \\ \hline $\mathcal{C}_{117}$ &
$(0101111110001110100001110110011011)$  & $\textbf{9}$ &
$\textbf{204}$   & $\mathcal{C}_{118}$ &
$(0111110111111110111101000001110100)$  & $\textbf{9}$ &
$\textbf{205}$   \\ \hline $\mathcal{C}_{119}$ &
$(1110110111101011000110100111111100)$  & $\textbf{9}$ &
$\textbf{206}$   & $\mathcal{C}_{120}$ &
$(0000000111010010100010010001011001)$  & $\textbf{9}$ &
$\textbf{207}$   \\ \hline $\mathcal{C}_{121}$ &
$(0001001101110101011111001000101101)$  & $\textbf{9}$ &
$\textbf{208}$   & $\mathcal{C}_{122}$ &
$(0100001111001011001010000111010011)$  & $\textbf{9}$ &
$\textbf{209}$   \\ \hline $\mathcal{C}_{123}$ &
$(0101000110111111010111000111000100)$  & $\textbf{9}$ &
$\textbf{210}$   & $\mathcal{C}_{124}$ &
$(0110110111011011011110111101001100)$  & $\textbf{9}$ &
$\textbf{211}$   \\ \hline $\mathcal{C}_{125}$ &
$(0001110001110001001001110010111010)$  & $\textbf{9}$ &
$\textbf{213}$   & $\mathcal{C}_{126}$ &
$(0001010100001110010110011101111101)$  & $\textbf{9}$ &
$\textbf{214}$   \\ \hline $\mathcal{C}_{127}$ &
$(0101010110001011110111000001101110)$  & $\textbf{9}$ &
$\textbf{215}$   & $\mathcal{C}_{128}$ &
$(0010011111011010100011110101011011)$  & $\textbf{9}$ &
$\textbf{216}$   \\ \hline $\mathcal{C}_{129}$ &
$(0111100100111001111101100111110101)$  & $\textbf{9}$ &
$\textbf{217}$   & $\mathcal{C}_{130}$ &
$(1101110110110011011001111111011011)$  & $\textbf{9}$ &
$\textbf{218}$   \\ \hline $\mathcal{C}_{131}$ &
$(1101011010110011000111101000101100)$  & $\textbf{9}$ &
$\textbf{219}$   & $\mathcal{C}_{132}$ &
$(0110101110111110101011011111101011)$  & $\textbf{9}$ &
$\textbf{220}$   \\ \hline $\mathcal{C}_{133}$ &
$(1110000011001101000110000000101110)$  & $\textbf{9}$ &
$\textbf{222}$   & $\mathcal{C}_{134}$ &
$(1001111010110000000101110100000100)$  & $\textbf{9}$ &
$\textbf{223}$   \\ \hline $\mathcal{C}_{135}$ &
$(1001000111100111010011111100111001)$  & $\textbf{9}$ &
$\textbf{224}$   & $\mathcal{C}_{136}$ &
$(1011111011110111101111011111011100)$  & $\textbf{9}$ &
$\textbf{225}$   \\ \hline $\mathcal{C}_{137}$ &
$(0011111100110101110101101110110101)$  & $\textbf{9}$ &
$\textbf{226}$   & $\mathcal{C}_{138}$ &
$(1011010011100011110000011000001011)$  & $\textbf{9}$ &
$\textbf{228}$   \\ \hline $\mathcal{C}_{139}$ &
$(0101001011001111001010011001000011)$  & $\textbf{9}$ &
$\textbf{230}$   &
 &    &  &    \\ \hline
\end{tabular}}
\end{center}
\end{table}

\section{Conclusion}
We introduced the concept of a distance between self-dual codes.
This generalizes the notion of a neighbor in self-dual codes, which
leads to a new way of constructing new self-dual codes from a known
one. Applying these ideas to an extremal binary self-dual code of
length 68 we were able to construct 143 new extremal binary
self-dual codes of length 68 with new weight enumerators, including
the first examples with $\gamma=8, 9$ in $W_{68,2}$ in the
literature. Thus, we have now completed the theoretical list of
possible $\gamma$ values that can be found in $W_{68,2}$. Generator
matrices for some of the new codes are available online at
\cite{web}. 42 of the codes we have constructed have $\gamma=8$ in
their weight enumerator, while 40 of them have $\gamma=9$ in their
weight enumerators. In particular, we have been able to construct
the codes that have the following parameters:

\begin{equation*}
\begin{split}
(\gamma =5,& \quad \beta =\{195,198,200,202,211\}), \\
(\gamma =6,& \quad \beta =\{175,177,179,181,182,183,185,186,187,188,189,190,191,193,194,195,196,197,198,199,\\
& \qquad \quad \; \,       200,201,202,204,206,207\}), \\
(\gamma =7,& \quad \beta =\{163,166,169,170,171,172,173,174,176,177,178,179,180,181,183,184,185,186,187,188,\\
& \qquad \quad \;  \,       190,191,192,193,194,195,197,198,199,212\}), \\
(\gamma =8,& \quad \beta =\{180,181,182,183,184,185,186,187,188,189,190,191,192,193,194,195,196,197,198,199,\\
& \qquad \quad \; \,        200,201,202,203,204,205,206,207,208,209,210,211,212,213,214,215,216,217,218,219,\\
& \qquad \quad \; \,        220,221\}), \\
(\gamma =9,& \quad \beta =\{186,187,188,189,190,192,193,194,195,196,198,199,200,201,202,203,204,205,206,207, \\
& \qquad \quad \; \,        208,209,210,211,213,214,215,216,217,218,219,220,221,222,223,224,225,226,228,230 \})\\
\end{split}%
\end{equation*}

The strength of this new approach has been demonstrated by the
number of new weight enumerators that we have been able to obtain by
applying it to a single code. We believe this will open up new
venues in the search and classification of new extremal binary
self-dual codes.


\begin{thebibliography}{99}
\bibitem{anev} D. Anev, M. Harada and N. Yankov, \textquotedblleft New
extremal binary self-dual codes of lengths 64 and 66", \emph{J. Alg.
Comb. Disc. Structures and Appl.}, vol. 5, no.3, pp. 143--151, 2018.

\bibitem{magma} W. Bosma, J. Cannon, and C. Playoust, ``The Magma algebra
system. I. The user language", \emph{J. Symbolic Comput}., Vol. 24,
pp. 235--265, 1997.

\bibitem{brualdi} R.A. Brualdi and V.S. Pless, ``Weight Enumerators of Self-Dual Codes", \emph{IEEE Trans. Infrom. Theory}, vol. 37, no. 4, pp. 1222-1225, 1991.

\bibitem{buyuklieva} S. Buyukl\i eva, I. Boukl\i ev, \textquotedblleft
Extremal self-dual codes with an automorphism of order 2",
\emph{IEEE Trans. Inform. Theory}, Vol. 44, pp. 323--328, 1998.

\bibitem{conway} J.H. Conway, N.J.A. Sloane, \textquotedblleft A new upper
bound on the minimal distance of self-dual codes", \emph{\ IEEE
Trans. Inform. Theory}, Vol. 36, 6, 1319--1333, 1990.

\bibitem{dougherty1} S.T. Dougherty, T.A. Gulliver, M.
Harada,\textquotedblleft Extremal binary self dual codes",
\emph{IEEE Trans. Inform. Theory}, Vol. 43 pp. 2036-2047, 1997.

\bibitem{QR} J. Gildea, H. Hamilton, A. Kaya and B. Yildiz, ``Modified quadratic residue constructions and new extremal binary self-dual codes of lengths 64, 66 and 68", \emph{Inf. Proces. Letters}, vol. 157, 2020.

\bibitem{GKTY} J. Gildea., A. Kaya, R. Taylor, and B. Yildiz, ``Constructions for self-dual codes induced from group rings",  \emph{Finite Fields Appl.}, vol. 51, pp. 71--92, 2018.

\bibitem{JoeAbidinAdrian} J. Gildea, A. Korban and A. Kaya, ``Self-dual codes using bisymmetric designs and group rings", {\bf submitted}.

\bibitem{alteredFC} J. Gildea, A. Kaya and B. Yildiz, ``An altered four circulant construction for self-dual codes from group rings and new extremal binary self-dual codes I", \emph{Discrete Mathematics.}, vol. 342, no. 12, 111620, 2019.

\bibitem{web} J. Gildea, A. Kaya, A. Korban, B. Yildiz,
\textquotedblleft Binary generator matrices for some new extremal
binary self-dual codes of length $68$", available online at
\url{http://abidinkaya.wixsite.com/math/neighbor}.

\bibitem{harada} M. Harada, A. Munemasa, \textquotedblleft Some restrictions
on weight enumerators of singly even self-dual codes", \emph{IEEE
Trans. Inform. Theory}, Vol. 52, pp. 1266--1269, 2006.

\bibitem{kaya} A. Kaya and B. Yildiz \textquotedblleft Various constructions for self-dual codes over rings and new
binary self-dual codes", \emph{%
Discrete Math}, Vol. 339, No. 2 pp. 460--469, 2016.

\bibitem{pasa} A. Kaya, B. Yildiz, A. Pa\c{s}a \textquotedblleft New
extremal binary self-dual codes from a modified four circulant
construction", \emph{Discrete Math}, Vol. 339, No. 3 pp. 1086--1094,
2016.

\bibitem{Rains} E.M. Rains, \textquotedblleft Shadow Bounds for Self Dual
Codes", \emph{IEEE Trans. Inf. Theory}, Vol.44, pp.134--139, 1998.


\end{thebibliography}
\end{document}